\documentclass[12pt]{amsart}

\usepackage{amsmath}
\usepackage{amsfonts}
\usepackage{amssymb}
\usepackage{amsthm}
\usepackage{mathptmx}
\usepackage{mathrsfs}
\usepackage{latexsym}
\usepackage{times}
\usepackage[mathscr]{euscript}
\usepackage[isolatin]{inputenc}
\usepackage[pagebackref]{hyperref}
\usepackage[backrefs,alphabetic,initials]{amsrefs}

\DeclareMathAlphabet{\mathpzc}{OT1}{pzc}{m}{it}

\newtheorem{thm}{Theorem}[section]
\newtheorem{lem}[thm]{Lemma}
\newtheorem{prop}[thm]{Proposition}
\newtheorem{cor}[thm]{Corollary}

\theoremstyle{definition}
\newtheorem{defn}[thm]{Definition}
\newtheorem{ex}[thm]{Example}

\theoremstyle{remark}
\newtheorem{rem}[thm]{Remark}

\usepackage{color}

\newcommand{\slk}{\mathfrak{sl}}

\newcommand{\ok}{\mathfrak{o}}

\newcommand{\g}{\mathfrak{g}}

\newcommand{\hk}{\mathfrak{h}}

\newcommand{\ak}{\mathfrak{a}}

\newcommand{\qk}{\mathfrak{q}}

\newcommand\Bc{\mathcal{B}}

\newcommand\FF{\mathbb F}

\newcommand{\Zs}{\mathscr Z}

\newcommand{\ad}{\operatorname{ad}}

\newcommand\id{\textrm{id}}

\renewcommand\hat\widehat
\renewcommand\tilde\widetilde 
\newcommand{\spa}{\operatorname{span}}

\newcommand{\OO}{\operatorname{O}}

\newcommand{\oplusp}{{ \ \overset{\perp}{\mathop{\oplus}} \ }}

\begin{document}

\date{\today}

\title[Solvable quadratic Lie algebras of dimensions $\leq 8$]{Solvable quadratic Lie algebras of dimensions $\leq 8$}

\author{Minh Thanh Duong}
\author{Rosane Ushirobira}

\address{Minh Thanh Duong, Department of Physics, Ho Chi Minh city
  University of Pedagogy, 280 An Duong Vuong, Ho Chi Minh city,
  Vietnam.}  \address{Non-A team, Inria, France \& Institut de
  Math\'ematiques de Bourgogne, Universit\'e de Bourgogne, B.P. 47870,
  F-21078 Dijon Cedex} \email{thanhdmi@hcmup.edu.vn}
\email{Rosane.Ushirobira@inria.fr}



\date{\today}
\maketitle
\begin{abstract}
  In this paper, we classify solvable Lie algebras of dimensions $\leq
  8$ endowed with a nondegenerate invariant symmetric bilinear form
  over an algebraically closed field. This classification (up to
  isometrically isomorphisms) is mainly based on the double extension method.
\end{abstract}

\section{Introduction}
Let us consider $\g$ a quadratic Lie algebra, that is, $\g$ is
equipped with a nondegenerate invariant symmetric bilinear form $B$. A
well-known problem in the theory of quadratic Lie algebras over an
algebraically closed field $\FF$ is to establish their
classification. This problem was intensively studied in many
works. For instance, it is initiated for nilpotent quadratic Lie
algebras of dimension $\leq$ 7 in \cite{FS87} and more recently, in
\cite{Kat07} the case of real nilpotent quadratic Lie algebras of
dimension $\leq$ 10 was examined. For the nonsolvable case, real
quadratic Lie algebras up to dimension 9 were analyzed in \cite{CS08}
and up to dimension 13 in a recent work \cite{BE14}. However, the
classification of {\bf solvable nonnilpotent} quadratic Lie algebras
is not trivial and it is only known in dimensions $\leq$ 6
\cite{BK03}. This was our motivation to attempt such a classification
for higher-dimensional Lie algebras. With this aim, in this work we
provide a classification of solvable quadratic Lie algebras of
dimensions $\leq 8$.
 The main method used here is the double
extension procedure defined in \cite{MR85}. 
	

The paper is organized as follows. In Section 1, we recall basis
definitions and give some useful results. We also present the
classification of solvable quadratic Lie algebras of dimensions $\leq
6$. In Section 2, the classification of solvable quadratic Lie
algebras of dimension 7 is obtained, by using the double extension of
5-dimensional solvable quadratic Lie algebras. Section 3 is dedicated
to the case when dimension is 8, by applying the double extension
procedure for 6-dimensional solvable quadratic Lie algebras.

\section{Quadratic Lie algebras of dimensions $\leq 6$}

\begin{defn} 
A Lie algebra $\g$ is called a {\em quadratic Lie algebra} if it is
endowed with a nondegenerate invariant (i.e. $B([X,Y],Z) =
B(X,[Y,Z])$ for all $X,\ Y,\ Z\in\g$) symmetric bilinear form $B$.
\end{defn}

It is easy to check that if $I$ is an ideal of $\g$ then $I^\bot$ is
also an ideal of $\g$. Moreover, if $I$ is nondegenerate (i.e. $B|_{I
  \times I}$ is nondegenerate), so is $I^\bot$ and $\g = I\oplus
I^\bot$. In this case, we use the convenient notation $\g = I\oplusp
I^\bot$. This implies that in order to study quadratic Lie algebras
one may focus on {\em indecomposable} ones. Recall that a quadratic
Lie algebra $\g$ is called indecomposable if it does not contain any
proper ideal which is nondegenerate. 
Otherwise, we call $\g$ {\em decomposable}.
	
Clearly, if $\g$ has a nonzero central element $X$ that is not
isotropic (i.e. $B(X,X)\neq 0$), then $\g$ is decomposable. A vector
subspace $V$ of $\g$ is called {\em totally isotropic} if $B(X,Y)=0$
for all $X,\ Y\in V$. Thus, for indecomposable quadratic Lie algebras,
the center $\Zs(\g)$ needs to be totally isotropic, or equivalently
$\Zs(\g)\subset[\g,\g]$ (we call such quadratic Lie algebras {\em reduced}).
	
\begin{defn}
Two quadratic Lie algebras $(\g,B)$ and $(\g',B')$ are {\em
  isometrically isomorphic} (or {\em i-isomorphic}, for short) if
there exists a Lie algebra isomorphism $A$ from $\g$ onto $\g'$
satisfying $B'(A(X), A(Y)) = B(X,Y)$ for all $X,\ Y \in \g$. In this
case, $A$ is called an {\em i-isomorphism}.
\end{defn}



\begin{defn}\label{defn1.2}
Let $(\hk,[\cdot, \cdot]_\hk,B)$ be a quadratic Lie algebra and $D$ a
{\em skew-symmetric} derivation of $\hk$ (i.e. $D$ satisfies
$B(D(X),Y) = -B(X,D(Y))$ for all $X,\ Y\in\hk$). We define on the
vector space $\g=\hk\oplus \FF e\oplus \FF f$ the product:
\[[X,Y]= [X,Y]_{\hk} +B(D(X),Y)f, \;  [e,X]=D(X), \ \forall 
X,\ Y\in\hk \ \text{ and } \; [f,\g]=0.\] Then $\g$ is a quadratic Lie
algebra with invariant bilinear form $B_{\g}$ defined by:
$$B_{\g}(e,e)=B_{\g}(f,f)=B_{\g}(e,\hk)=B_{\g}(f,\hk)=0, \ B_{\g}(X,Y)=B(X,Y)
$$ and $B_{\g}(e,f)=1$ for all $X,\ Y\in\g$. We call $\g$ the {\em
  double extension of $\hk$ by means of $D$} or a {\em one-dimensional
  double extension of $\hk$}, for short. We may use the notation $(\g,
B_\g,D)$.
\end{defn}

Double extensions are a very useful procedure to obtain quadratic Lie
algebras from given ones and they appear frequently in the
classification problem. In the above Definition, if $\hk$ is Abelian and $D\neq 0$
then it is obvious that $\left[[\g,\g],[\g,\g]\right]$ is at most
one-dimensional. In this case, $\g$ is called a {\em 1-step double
  extension} or a {\em singular} quadratic Lie algebra as described in
\cite{DPU12}. All 1-step double extensions were completely classified
in \cite{DPU12} and their relation to the classification of
$\OO(n)$-adjoint orbits in $\ok(n)$ was established. Moreover, a
remarkable result in \cite{FS87} asserts that one-dimensional double
extensions are sufficient for studying solvable quadratic Lie
algebras.

Next, we provide some classification criteria that are useful for our
work.
\begin{lem}\label{lem1}
Let $(\g_1,B_{\g_1}, D_1)$ and $(\g_2,B_{\g_2},D_2)$ be double
extensions of $(\hk,B)$ by means of $D_1$ and $D_2$ respectively. We write $\g_1=\hk\oplus\FF e\oplus \FF f$ and $\g_2=\hk\oplus\FF
e'\oplus \FF f'$ as double extensions. Then there is an i-isomorphism $A:\g_1\rightarrow \g_2$ satisfying $A(\FF f)=\FF f'$ if and only if there exist an i-isomorphism
$P:\hk\rightarrow\hk$, a nonzero $\lambda\in\FF$ and an element
$X\in\hk$ such that $P^{-1}D_1P=\lambda D_2 +\ad(X)$.
\end{lem}

\begin{proof}
Let $A:\g_1\rightarrow \g_2$ be an i-isomorphism satisfying $A(\FF f)=\FF f'$ then there is a nonzero $\alpha\in\FF$ such that $A(f)=\alpha f'$. Since $B_{\g_2}(A(\hk),A(f))=0$, one has $A(\hk)\subset f'^\bot=\hk\oplus\FF f'$. Hence, 
we can decompose $A(Y)=A_0(Y)+\phi(Y)f'$ for all $Y\in\hk$
where $A_0:\hk\rightarrow \hk$ and $\phi:
\hk\rightarrow\FF$. 
It is
easy to check that $A_0$ is also an i-isomorphism of $\hk$.

We assume that $A(e)=\lambda e' + X + \gamma f'$ for some $X \in \hk$,
$\lambda$, $\gamma \in \FF$. Since $A$ is an isometry,
$B_{\g_2}(A(e),A(f))=1$ and $B_{\g_2}(A(e),A(e))=0$, so one has
$\lambda=\frac{1}{\alpha}$ and $\gamma=-\frac{B(X,X)}{2\lambda}$. For
all $Y\in\hk$, since $0=B_{\g_1}(Y,e)=B_{\g_2}(A(Y),A(e))$ one has
$\phi(Y)=-\frac{1}{\lambda} B(X,A_0(Y))$. Moreover, $$A \left( [e,Y]
\right) = A(D_1(Y))=A_0(D_1(Y))+\phi(D_1(Y))f'$$
and $$[A(e),A(Y)]=\lambda
D_2(A_0(Y))+\ad(X)(A_0(Y))+B(D_2(X),A_0(Y))f'$$ then
$A_0(D_1(Y))=\lambda D_2(A_0(Y))+\ad(X)(A_0(Y))$. It means
$A_0D_1=\lambda D_2 A_0+\ad(X) A_0$. In this case, if we set
$P=A_0^{-1}$ then $P^{-1}D_1P=\lambda D_2 +\ad(X)$.

Conversely, if there exist an i-isomorphism $P:\hk\rightarrow\hk$, a
nonzero $\lambda\in\FF$ and an element $X\in\hk$ such that
$P^{-1}D_1P=\lambda D_2 +\ad(X)$ then we define
$A:\g_1\rightarrow\g_2$ by: $A(e)=\lambda e' +
X-\frac{B(X,X)}{2\lambda}f'$, $A(Y)=P^{-1}(Y)-\frac{1}{\lambda}
B(X,P^{-1}(Y))f'$ and $A(f)=\frac{1}{\lambda}f'$. It is
straightforward to check that $A$ is an i-isomorphism.
\end{proof}

Since i-isomorphisms and invertible derivations preserve the center
$\Zs(\hk)$ of $\hk$ then we have the following corollary.

\begin{cor}\label{cor1}
Let $(\g_1,B_{\g_1}, D_1)$ and $(\g_2,B_{\g_2},D_2)$ be double extensions of $(\hk,B)$
where $D_1$ and $D_2$ are invertible derivations. Assume $\g_1$ and
$\g_2$ are i-isomorphic. Consider $D_1^c=D_1|_{\Zs(\hk)}$ and
$D_2^c=D_2|_{\Zs(\hk)}$. Then $D_1^c$ coincides with $D_2^c$ up to a
conjugation and up to nonzero multiples. As a consequence, if $D_1^c$
is diagonalizable then so is $D_2^c$.
\end{cor}

\begin{defn}
Let $\qk$ be a vector space of dimension $m$ endowed with a
nondegenerate symmetric bilinear form $B$. In this case, $(\qk,B)$ is
called a {\em quadratic vector space}. We call a {\em canonical basis}
$\Bc = \{ X_1, \dots, X_m \}$ of $\qk$ if it satisfies the following
conditions: if $m$ even, $m = 2n$, write $$\Bc = \{X_1, \dots, X_n,
Z_1, \dots, Z_n\},$$ if $m$ is odd, $m = 2n+1$, write $$\Bc = \{X_1,
\dots, X_n, T, Z_1, \dots, Z_n \}.$$ One has:

\begin{itemize}

\item if $m = 2n$ then $B(X_i, Z_j) =\delta_{ij},\
B(X_i, X_j) = B(Z_i, Z_j) = 0$ where $1 \leq
i,j \leq n$.
\item if $m = 2n+1$ then $B(X_i, Z_j) = \delta_{ij}$, $B(X_i, X_j) = B(Z_i,
    Z_j) = B(X_i, T) = B(Z_j, T) =
    0$, $B(T,T) = 1$ where $1 \leq i,j \leq n$.
\end{itemize}
\end{defn}

\begin{defn}\label{defn1.6}
Let $(\qk, B_{\qk})$, $(\hk, B_{\hk})$ be a quadratic vector
spaces. Assume that $\hk$ is 4-dimensional. Let $\{X_1,X_2,Z_1,Z_2\}$
be a canonical basis of $\hk$ and $C_1$, $C_2\in\ok(\qk)$ satisfying
$C_1C_2=C_2C_1$. We define on $\g:=\qk\oplus\hk$ the product as
follows: $$[X_i,X]=C_i(X),\ [X,Y]=B_{\qk}(C_1(X),Y)Z_1+B_{\qk}(C_2(X),Y)Z_2$$
for all $X,\ Y\in\qk$ and $[\hk,\hk]=0$. Then $\g$ becomes a quadratic
Lie algebra with invariant bilinear form $B=B_{\qk}+B_{\hk}$ and it is
called a {\em double extension of $\qk$ by the pair $(C_1,C_2)$}.
\end{defn}

\begin{prop}\label{prop2}
Let $\g$ and $\g'$ be double extensions of $(\qk,B_{\qk})$ by the
pairs $(C_1,C_2)$ and $(C_1',C_2')$ respectively. We assume $C_1$ and
$C_1'$ are invertible. Then $\g$ and $\g'$ are i-isomorphic if and only
if there exist an isometry $P$ of $\qk$ and nonzero pairs
$(\lambda_1,\lambda_2)$, $(\gamma_1,\gamma_2)$ in $\FF^2$ such that
$\lambda_1\gamma_2-\lambda_2\gamma_1 \neq 0$,
$$PC_1P^{-1}=\lambda_1C_1'+\lambda_2C_2'\ \ \text{and}\ \ PC_2P^{-1}=\gamma_1C_1'+\gamma_2C_2'.$$
\end{prop}

\begin{proof}
We write $\g=\qk\oplus\hk$ and $\g'=\qk\oplus\hk'$ as double
extensions by the pairs $(C_1,C_2)$ and $(C_1',C_2')$ respectively
where $\hk=\spa\{X_1,X_2,Z_1,Z_2\}$,
$\hk'=\spa\{X_1',X_2',Z_1',Z_2'\}$, $[X_i,X]=C_i(X)$ and
$[X_i',X]=C_i'(X)$ for all $X\in\qk$. Denote by $B_{\g}$ and $B_{\g'}$
the invariant symmetric bilinear forms on $\g$ and $\g'$ respectively.

If $A:\g\rightarrow \g'$ is an i-isomorphism then it maps $\hk\oplus\FF Z_1\oplus \FF Z_2$ onto $\hk\oplus\FF Z_1'\oplus \FF Z_2'$ since $C_1$, $C_1'$ are invertible . Moreover, one has $\Zs(\g)=\FF Z_1\oplus\FF Z_2$ and $\Zs(\g')=\FF Z_1'\oplus\FF Z_2'$. Therefore,  we can assume 
$$A(X_1)=\lambda_1X_1'+\lambda_2X_2'+T_1+\alpha_1Z_1'+\alpha_2Z_2',$$ $$A(X_2)=\gamma_1X_1'+\gamma_2X_2'+T_2+\beta_1Z_1'+\beta_2Z_2',$$
$A(X)=P(X)+\phi_1(X)Z_1'+\phi_2(X)Z_2'$ for all $X\in\qk$, and $A(Z_1)=a_1Z_1'+a_2Z_2'$, $A(Z_2)=b_1Z_1'+b_2Z_2'$ where $T_1,T_2\in\qk$, $P:\qk\rightarrow\qk$ and $\phi_1,\phi_2:\qk\rightarrow\FF$. Since $B_{\qk}(P(X),P(Y))=B_{\g'}(A(X),A(Y))=B_{\qk}(X,Y)$ for all $X,Y\in\qk$ then $P$ is an isometry of $\qk$. On the other hand, $B_{\g'}(A(X_1),A(Z_1)) = 1$ and $B_{\g'}(A(X_2),A(Z_1)) = 0$ then the determinant 
$\left|\begin{matrix} \lambda_1 & \lambda_2  \\ \gamma_1 & \gamma_2\end{matrix}\right|=\lambda_1\gamma_2-\lambda_2\gamma_1$ must be nonzero. By checking the conditions $A[X_1,X]=[A(X_1),A(X)]$ and $A[X_2,X]=[A(X_2),A(X)]$ we obtain $PC_1P^{-1}=\lambda_1C_1'+\lambda_2C_2'$ and $PC_2P^{-1}=\gamma_1C_1'+\gamma_2C_2'$.

Conversely, if there exist an isometry $P$ of $\qk$ and nonzero pairs $(\lambda_1,\lambda_2)$, $(\gamma_1,\gamma_2)$ in $\FF^2$ such that $\lambda_1\gamma_2-\lambda_2\gamma_1 \neq 0$, $PC_1P^{-1}=\lambda_1C_1'+\lambda_2C_2'$ and $PC_2P^{-1}=\gamma_1C_1'+\gamma_2C_2'$ then we define 
$A(X_1)=\lambda_1X_1'+\lambda_2X_2'$, $A(X_2)=\gamma_1X_1'+\gamma_2X_2'$, $A(X)=P(X)$ for all $X\in\qk$, and $A(Z_1)=a_1Z_1'+a_2Z_2'$, $A(Z_2)=b_1Z_1'+b_2Z_2'$ where $a_i,b_i$ satisfy the conditions: $$\left\{\begin{matrix} \lambda_1a_1+\lambda_2a_2=1 \\ \gamma_1a_1+\gamma_2a_2= 0  \\ \lambda_1b_1+\lambda_2b_2=0 \\ \gamma_1b_1+\gamma_2b_2= 1\end{matrix}\right. .$$ We can check that $A$ is an i-isomorphism from $\g$ onto $\g'$.
\end{proof}

Definition \ref{defn1.6} and Proposition \ref{prop2} are enough for
our purpose in the classification of solvable quadratic Lie algebras
of low dimensions. But for a larger view about quadratic Lie algebras
having $[[\g,\g],[\g,\g]]\subset\Zs(\g)$, we introduce another type as
in the below example.

\begin{ex}
Let $\ak=\qk\oplusp\FF T$, $\hk$ be quadratic vector spaces with
nondegenerate symmetric bilinear forms $B_{\ak}$ and $B_{\hk}$
respectively. We assume $\hk$ is 4-dimensional and
$B_{\ak}(T,T)=1$. Let $\{X_1,X_2,Z_1,Z_2\}$ be a canonical basis of
$\hk$, and $C_1$, $C_2\in\ok(\qk)$ satisfying $C_1C_2=C_2C_1$. We
define on $\g:=\ak\oplus\hk$ the product as
follows: $$[X_1,X_2]=T,\ [X_1,T]=-Z_2,\ [X_2,T]=Z_1,\ [X_i,X]=C_i(X),$$
$[X,Y]=B_{\ak}(C_1(X),Y)Z_1+B_{\ak}(C_2(X),Y)Z_2$ for all
$X,\ Y\in\qk$ and $[T,\qk]=0$. Then $\g$ becomes a quadratic Lie
algebra with invariant bilinear form $B=B_{\ak}+B_{\hk}$ and it is
called a {\em double extension of $\qk$ by the triple $(T,C_1,C_2)$}.
\end{ex}

Finally we recall the classification of solvable quadratic Lie algebras of dimension $\leq 6 $ is given in \cite{BK03} as follows.

\begin{prop}\label{prop1} Let $(\g,B)$ be a solvable quadratic Lie algebra of dimension $n\leq 6$.

\begin{enumerate}
	\item If $n\leq 3$ then $\g$ is i-isomorphic to $\FF^n$.
	\item If $n=4$ then $\g$ is i-isomorphic to $\FF^4$ or the diamond Lie algebra $\g_4$ where $\g_4$ is the double extension of $\FF^2$ by means of $D=\begin{pmatrix} 1 & 0  \\ 0 & -1 \end{pmatrix}$ in a canonical basis of $\FF^2$.
	\item If $n=5$ then $\g$ is i-isomorphic to $\FF^5$, $\g_4\oplusp\FF$ or $\g_5$ where $\g_5$ is the double extension of $\FF^3$ by means of $D=\begin{pmatrix} 0 & 1  & 0 \\ 0 & 0 & -1 \\ 0 & 0 & 0\end{pmatrix}$ in a canonical basis of $\FF^3$.
	\item If $n=6$ then $\g$ is i-isomorphic to $\FF^6$, $\g_4\oplusp\FF^2$, $\g_5\oplusp\FF$ or the quadratic Lie algebras $\g_{6,1}$, $\g_{6,2}(\lambda)$ and $\g_{6,3}$ which are the double extensions of $\FF^4$ by means of 
$$D=\begin{pmatrix} 0 & 1  & 0 & 0\\ 0 & 0 & 0 & 0 \\ 0 & 0 & 0 & 0\\ 0 & 0 & -1 & 0\end{pmatrix},\ \begin{pmatrix} 1 & 0  & 0 & 0\\ 0 & \lambda & 0 & 0 \\ 0 & 0 & -1 & 0\\ 0 & 0 & 0 & -\lambda\end{pmatrix}\ \text{and}\ \begin{pmatrix} 1 & 1  & 0 & 0\\ 0 & 1 & 0 & 0 \\ 0 & 0 & -1 & 0\\ 0 & 0 & -1 & -1\end{pmatrix},$$ respectively in a canonical basis of $\FF^4$ with $0<\left|\lambda\right|\leq 1$. In this case, $\g_{6,2}(\lambda_1)$ and $\g_{6,2}(\lambda_2)$ is i-isomorphic if and only if $\lambda_2=\pm\lambda_1$ or $\lambda_2=\pm\lambda_1^{-1}$.
\end{enumerate}

\end{prop}

For the non solvable case of dimension $n\leq 6$, the reader is referred to \cite{CS08} or \cite{BE14} where we have five Lie algebras $\slk_2(\FF)$, $\slk_2(\FF)\oplusp\FF^i$, $1\leq i\leq 3$ and $T^*_0\left(\slk_2(\FF)\right)$ the semidirect product of $\slk_2(\FF)$ and its dual space by the coadjoint representation.
\section{Solvable quadratic Lie algebras of dimension 7}\label{sec2}

In this section, we shall give a complete classification of 7-dimensional solvable quadratic Lie algebras. First we recall Lemma 5.1 in \cite{Med85} as follows.
\begin{lem}\label{lem2.1}
Let $\g$ be a quadratic Lie algebra and $D_1,\ D_2$ skew-symmetric derivations of $\g$. If $D_1-D_2=\ad(X)$ with $X\in\g$ then double extensions of $\g$ by means of $D_1$ and $D_2$ are i-isomorphic.
\end{lem}

Remark that this is a particular case of Lemma \ref{lem1} where $P=\id$ and $\lambda=1$. Another useful and straightforward result is that double extensions of
$\g$ by means of $D$ and $\lambda D$ with nonzero $\lambda$ are i-isomorphic.

\begin{prop}\label{prop3} Let $\g$ be a solvable quadratic Lie algebra of dimension $7$.

\begin{enumerate}
	\item If $\g$ is decomposable then $\g$ is i-isomorphic to
       $\g_6\oplusp\FF$ where $\g_6$ is a solvable quadratic Lie
       algebra of dimension 6 given in Proposition \ref{prop1} (4).

	\item If $\g$ is indecomposable then $\g$ is i-isomorphic to each
       of the following quadratic Lie algebras:
	\begin{enumerate}
	\item[(i)] $\g_{7,1}$ and $\g_{7,2}$: the double extensions of
       $\FF^5$ by means of
$$D = \begin{pmatrix} 0 & 1 & 0 & 0 & 0 \\ 0 & 0 & 1 & 0 & 0 \\ 0 & 0
       & 0 & 0 & -1 \\ 0 & 0 & 0 & 0 & 0 \\ 0 & 0 & 0 & -1 &
       0 \end{pmatrix}\ \text{and}\ \begin{pmatrix} 1 & 0 & 0 & 0 & 0
       \\ 0 & 0 & 1 & 0 & 0 \\ 0 & 0 & 0 & 0 & -1 \\ 0 & 0 & 0 & -1 &
       0 \\ 0 & 0 & 0 & 0 & 0 \end{pmatrix},$$ respectively in a
       canonical basis of $\FF^5$,
	
	\item[(ii)] $\g_{7,3}$: the double extension of $\g_5$ by means
       of $$D=\begin{pmatrix} 1 & 0 & 0 & 0 & 0\\ 0 & -1 & 0 & 0 &
       0\\ 0 & 0 & 0 & 0 & 0\\ 0 & 0 & 0 & -1 & 0 \\ 0 & 0 & 0 & 0 &
       1\end{pmatrix}$$ in a canonical basis $\{X_1,X_2,T,Z_1,Z_2\}$
       of $\g_5$ with $[X_1,X_2]=T$, $[X_1,T]=-Z_2$ and $[X_2,T]=Z_1$.
\end{enumerate}
\end{enumerate}

\end{prop}
\begin{proof}

The statement (1) is obvious and it follows Proposition
\ref{prop1}. We assume $(\g,B)$ is an indecomposable solvable quadratic
Lie algebra of dimension 7. Then $\g$ contains an isotropic central
element $Z$ and, hence, there exists an isotropic element $X\in\g$ such that $B(Z,X)=1$. If we denote by $\hk=(\FF X\oplus\FF Z)^\bot$ the orthogonal component of $\FF X\oplus\FF Z$ with respect to $B$, with the Lie structure induced by the one of $Z^\bot/\FF Z$ then we can check that $\hk$ with the bilinear form $B_\hk:=B|_{\hk\times\hk}$ is a solvable quadratic Lie algebra of dimension 5 and $\g$ is the double extension of $\hk$ by means of the derivation $D=\ad(X)|_\hk$. Therefore we can begin with $\g=\hk\oplus \FF
e\oplus \FF f$ a double extension of $(\hk,B)$ by means of $D$ where
$\hk$ is a solvable quadratic Lie algebra of dimension 5 and $D$ is a
skew-symmetric derivation of $\hk$. From the classification of
solvable quadratic Lie algebras of dimension 5 we consider the
following cases:

\begin{itemize}
	\item $\hk=\FF^5$. In this case, $\g$ is a 1-step double extension and the classification of $\g$ follows the classification of linear maps $D\in\ok(5)$ with totally isotropic kernel up to conjugation and up to nonzero multiples (see details in \cite{DPU12}). So 
we obtain two corresponding Lie algebras given in (i).
\item $\hk=\g_4\oplusp\FF$. Choose a basis $\{X,P,Q,Z,Y\}$ of $\hk$ such that the nonzero brackets $[X,P]=P$, $[X,Q]=-Q$, $[P,Q]=Z$ and the nonzero bilinear forms $B(X,Z)=B(P,Q)=B(Y,Y)=1$.

It is straightforward to check that $D$ is a skew-symmetric derivation of $\hk$ if and only if $D$ has the matrix in the given basis as follows:

$$ D = \begin{pmatrix} 0 & 0 & 0 & 0 & 0 \\ y & x & 0 & 0 & 0 \\ z & 0 & -x & 0 & 0 \\ 0 & -z & -y & 0 & t \\ -t & 0 & 0 & 0 & 0 \end{pmatrix}.$$
where $x,\ y,\ z,\ t\in\FF$. By Lemma \ref{lem2.1}, we can assume that $x$, $y$, $z$ are zero and then 
it is easy to see that the double extension of $\hk$ by means of $D$ is a 1-step double extension.

\item $\hk=\g_5$. Choose a canonical basis $\{X_1,X_2,T,Z_1,Z_2\}$ of $\g_5$ with $[X_1,X_2]=T$, $[X_1,T]=-Z_2$, $[X_2,T]=Z_1$. 
We can check that $D$ is a skew-symmetric derivation of $\g_5$ if and only if the matrix of $D$ in the given basis is:
$$D = \begin{pmatrix} x & z & 0 & 0 & 0 \\ y & -x & 0 & 0 & 0 \\ -b & -c & 0 & 0 & 0 \\ 0 & -t & b & -x & -y \\ t & 0 & c & -z & x \end{pmatrix}$$
where $x,\ y,\ z,\ t,\ b,\ c\in\FF$. Combined with Lemma \ref{lem2.1}, we consider only $D$ with matrix:
$ D = \begin{pmatrix} A & 0 & 0 \\ 0 & 0 & 0 \\ 0 & 0 & -A^t \end{pmatrix}$ where $A=\begin{pmatrix} x & z\\ y & -x \end{pmatrix}$ and $A^t$ is the transpose of $A$. 

If $P$ is a $2\times 2$ matrix such that its determinant is 1. Set $\begin{pmatrix} X_1' & X_2' \end{pmatrix}=\begin{pmatrix} X_1 & X_2 \end{pmatrix}P$
and $\begin{pmatrix} Z_1' & Z_2' \end{pmatrix}=\begin{pmatrix} Z_1 & Z_2 \end{pmatrix}(P^t)^{-1}$. It is easy to check that $\{X_1',X_2',T,Z_1',Z_2'\}$ is still a canonical basis of $\g_5$. Moreover, $[X_1',X_2']=T$, $[X_1',T]=-Z_2$, $[X_2',T]=Z_1$. In this case, the matrix of $D$ in the basis $\{X_1',X_2',T,Z_1',Z_2'\}$ is given by:
$$D=\begin{pmatrix} P^{-1}AP & 0 & 0 \\ 0 & 0 & 0\\ 0 & 0 &  -P^tA^t(P^t)^{-1}\end{pmatrix}.$$

Since every matrix is similar to a matrix in Jordan form and $A$ has the zero trace so we only consider matrices $A=\begin{pmatrix} 0 & 0  \\ 0 & 0 \end{pmatrix}$, $\begin{pmatrix} 0 & 1  \\ 0 & 0 \end{pmatrix}$ and $\begin{pmatrix} 1 & 0  \\ 0 & -1 \end{pmatrix}$.

It is obvious that if $A=\begin{pmatrix} 0 & 0  \\ 0 & 0 \end{pmatrix}$ then $\g$ is decomposable. If $A=\begin{pmatrix} 0 & 1  \\ 0 & 0 \end{pmatrix}$. In this case, the Lie bracket on $\g$ is defined by:
$$
[e,X_2]=X_1,\ [e,Z_1]=-Z_2, \ [X_1,X_2]=T,\ [X_1,T]=-Z_2,
$$
$[X_2,T]=Z_1$ and $[X_2,Z_1] = f$. Then $\g=\qk\oplus\left( \FF X_2\oplus\FF Z_2\right)$ is the double extension of $\qk$ spanned by $\{e,X_1,T,f,Z_1\}$ by the skew-symmetric $C:\qk\rightarrow\qk$, $C(e)=-X_1$, $C(X_1)=-T$, $C(T)=Z_1$ and $C(Z_1)=f$. This quadratic Lie algebra is 1-step.

It remains $A=\begin{pmatrix} 1 & 0 \\ 0 & -1 \end{pmatrix}$. In this
case, $\left[\left[\g,\g\right],\left[\g,\g\right]\right]$ is
generated by $\{T,Z_1,Z_2,f\}$ so $\g$ is not a 1-step double
extension. Moreover, $\g$ is also indecomposable since otherwise $\g$
would be a 1-step double extension (recall that all non-Abelian solvable quadratic Lie
algebras up to dimension 6 are 1-step double extensions). This algebra
is $\g_{7,3}$ given in (ii) and the proposition is proved.
\end{itemize}
\end{proof}
%
%
\begin{rem}
In the above proof, we can begin with a reduced solvable quadratic Lie algebra
$\g$ of dimension 7
. Since
$[\g,\g]/\Zs(\g)$ is nilpotent then $[\g,\g]/\Zs(\g)$ is either
Abelian or i-isomorphic to $\g_5$. The case $[\g,\g]/\Zs(\g)$ Abelian
were classified in \cite{KO04} and \cite{KO06} by cohomology. In the
case $[\g,\g]/\Zs(\g)$ i-isomorphic to $\g_5$, it is obvious that
$\dim(\Zs(\g))=1$, $\g$ is a double extension of $\g_5$ by a skew-symmetric derivation invertible on the center of $\g_5$ and it can not
be a 1-step double extension.
\end{rem}

\section{Solvable quadratic Lie algebras of dimension 8}

In the following we will give a classification of solvable quadratic
Lie algebras of dimension 8. We apply again the double extension
method for each of six-dimensional solvable quadratic Lie algebras
given in Proposition \ref{prop1}. In particular, we will describe step
by step double extensions of $\g_4\oplusp\FF^2$, $\g_5\oplusp\FF$,
$\g_{6,1}$, $\g_{6,2}(\lambda)$ and $\g_{6,3}$.
\subsection{Double extensions of $\g_4\oplusp\FF^2$}\hfill

It is a straightforward computation to obtain that if $\hk=\g_4\oplusp\FF^n$, $n\geq 1$, choosing a basis $\{X,P,Q,Z\}$ of $\g_4$ such that the nonzero bilinear forms $B(X,Z)=B(P,Q)=1$, the nonzero brackets $[X,P]=P$, $[X,Q]=-Q$, $[P,Q]=Z$ and an orthonormal basis $\{X_i\}$ of $\FF^n$, $1\leq i\leq n$, then all skew-symmetric derivations of $\hk$ can be described by the following matrix with respect to the basis $\{X,P,Q,Z,X_i\}$:
$$D = \begin{pmatrix} 0 & 0 & 0 & 0 &  0 \\ y & x & 0 & 0 &  0 \\ z & 0 & -x & 0 &  0 \\ 0 & -z & -y & 0 & A \\ -A^t & 0 & 0 & 0 & C \end{pmatrix}$$
where $A=\begin{pmatrix} x_1 & x_2 & ... & x_n \end{pmatrix}$, $x,\ y,\ z,\ x_i\in\FF$ and $C\in\ok(n)$. By Lemma \ref{lem2.1}, we can suppose that $\ x,\ y,\ z$ are zero. In this case, if $P$ is an isometry of $\FF^n$ then $Q:=\id\otimes P$ with $\id$ the identity map of $\g_4$ is an i-isomorphism of $\hk$ and it operates on $D$ to the matrix
$\begin{pmatrix} 0 & 0 & 0 \\ 0 & 0 & AP \\ -P^tA^t & 0 & P^{-1}CP \end{pmatrix}$.

Now we consider a particular situation where $n=2$. By the classification of $\OO(2)$-adjoint orbit in $\ok(2)$, we have two following cases:

\begin{enumerate}
	\item $C= \begin{pmatrix} 0 & 0 \\  0 & 0  \end{pmatrix}$ then it is easy to check that the double extension of $\hk$ by means of $D$ is a 1-step double extension of $\FF^6$ generated by $\{e,P,Q,f,X_1,X_2\}$.
	\item $C= \begin{pmatrix} 0 & -1 \\  1 & 0  \end{pmatrix}$ where $\{X_1,X_2\}$ is an orthonormal basis of $\FF^2$. 
		Let $\g=\hk\oplus\FF e\oplus\FF f$ be the double extension of $\hk$ by means of $D$. Replacing  $X+x_2X_1-x_1X_2$ by $X$, $X_1-x_2 Z$ by $X_1$ and $X_2+x_1 Z$ by $X_2$, one has $\g=\spa\{e,X_1,X_2,f\}\oplusp\spa\{X,P,Q,Z\}$ decomposable.
\end{enumerate}

Hence, we can conclude that every one-dimensional double extension of $\g_4\oplusp\FF^2$ is 1-step or decomposable.
\subsection{Double extensions of $\g_5\oplusp\FF$}\label{subsec3.2}\hfill

If $\hk=\g_5\oplusp\FF^n$, $n\geq 1$, choosing a canonical basis $\{X_1,X_2,T,Z_1,Z_2\}$ of $\g_5$ such that  $[X_1,X_2]=T$, $[X_1,T]=-Z_2$, $[X_2,T]=Z_1$ and an orthonormal basis $\{Y_i\}$ of $\FF^n$, $1\leq i\leq n$, then all skew-symmetric derivations of $\g$ can be described by the following matrix with respect to the basis $\{X_1,X_2,T,Z_1,Z_2,Y_i\}$:
$$D = \begin{pmatrix} x & z & 0 & 0 & 0 & 0  \\ y & -x & 0 & 0 & 0 & 0 \\ -b & -c & 0 & 0 & 0 & 0 \\ 0 & -t & b & -x & -y & A\\ t & 0 & c & -z & x & B \\ -A^t & -B^t & 0 & 0 & 0 & C\end{pmatrix}$$
where $A=\begin{pmatrix} x_1 & x_2 & ... & x_n \end{pmatrix}$, $B=\begin{pmatrix} y_1 & y_2 & ... & y_n \end{pmatrix}$, $x,\ y,\ z,\ b,\ c,\ t,\ x_i,\ y_i\in\FF$ and $C\in\ok(n)$. Combining with Lemma \ref{lem2.1} where inner derivations can be eliminated we can assume that $b$, $c$ and $t$ are zero. 

For the case of dimension 8, one has $n=1$ and then we can write $\hk=\g_5\oplusp\FF Y$ and $D = \begin{pmatrix} -x & -z & 0 & 0 & 0 & 0  \\ -y & x & 0 & 0 & 0 & 0 \\ 0 & 0 & 0 & 0 & 0 & 0 \\ 0 & 0 & 0 & x & y & -\alpha\\ 0 & 0 & 0 & z & -x & -\beta \\ \alpha & \beta & 0 & 0 & 0 & 0\end{pmatrix}$
with respect to the basis $\{X_1,X_2,T,Z_1,Z_2,Y\}$ of $\hk$. Following the classification of double extensions of $\g_5$ in Section \ref{sec2}, we consider i-isomorphisms $P:\g_5\oplusp\FF\longrightarrow \g_5\oplusp\FF$ such that $P:=Q\otimes\id$ with $Q$ an i-isomorphism of $\g_5$ and $\id$ the identity map of $\FF$. As a consequence, we have three following cases: 
\begin{enumerate}
	\item $D=\begin{pmatrix} 0 & 0 & 0 & 0 & 0 & 0 \\ 0 & 0 & 0 & 0 &
       0 & 0 \\ 0 & 0 & 0 & 0 & 0 & 0 \\ 0 & 0 & 0 & 0 & 0 &
       -\alpha\\ 0 & 0 & 0 & 0 & 0 & -\beta \\ \alpha & \beta & 0 & 0
       & 0 & 0\end{pmatrix}$. If $\alpha=\beta=0$ then $\g$ is
       decomposable. If $\alpha\neq 0$ then we can choose
       $\alpha=1$. In this case, replacing $X_2-\beta X_1$ by $X_2$
       and $Z_1+\beta Z_2$ by $Z_1$, we obtain $\beta=0$.
	It is easy to check that $\g$ is a 1-step double extension of $\FF^6=\spa\{e,X_2,Y,T,f,Z_2\}$.
	\item $D=\begin{pmatrix} 0 & 1 & 0 & 0 & 0 & 0  \\ 0 & 0 & 0 & 0 & 0 & 0 \\ 0 & 0 & 0 & 0 & 0 & 0 \\ 0 & 0 & 0 & 0 & 0 & -\alpha\\ 0 & 0 & 0 & -1 & 0 & -\beta \\ \alpha & \beta & 0 & 0 & 0 & 0\end{pmatrix}$. If $\alpha=0$ then $\g$ is a 1-step double extension of $\FF^6=\spa\{e,X_1,Y,T,f,Z_1\}$.  
	If $\alpha\neq 0$ and $\beta=0$ then 
	the Lie bracket on $\g$ is defined by: $[e,X_1]=\alpha Y$, $[e,X_2]=X_1$, $[e,Z_1]=-Z_2$, $[e,Y]=-\alpha Z_1$, $[X_1,X_2]=T$, $[X_1,T]=-Z_2$, $[X_2,T]=Z_1$, $[X_2,Z_1]=f$ and $[X_1,Y]=\alpha f$. It is easy to check that $\dim([\g,\g])=6$, $\dim\left(\left[[\g,\g],[\g,\g]\right]\right)=2$ and this Lie algebra is 5-nilpotent. Therefore it must be indecomposable and not a 1-step double extension by following the classification in lower dimensions. We denote this quadratic Lie algebra by $\g_{8,1}(\alpha)$. Note that $\g_{8,1}(\alpha)$ and $\g_{8,1}(-\alpha)$ is i-isomorphic.
	
	If $\alpha$ and $\beta\neq 0$ then we set $X_2:=X_2-\frac{\beta}{\alpha}X_1$ and $Z_1:=Z_1+\frac{\beta}{\alpha}Z_1$ to turn to the previous case of $\alpha\neq 0$ and $\beta=0$.
	\item $D=\begin{pmatrix} 1 & 0 & 0 & 0 & 0 & 0  \\ 0 & -1 & 0 & 0 & 0 & 0 \\ 0 & 0 & 0 & 0 & 0 & 0 \\ 0 & 0 & 0 & -1 & 0 & -\alpha\\ 0 & 0 & 0 & 0 & 1 & -\beta \\ \alpha & \beta & 0 & 0 & 0 & 0\end{pmatrix}$. In this case, $\g$ is decomposable since $u=Y-\alpha Z_1+\beta Z_2$ is central and $B(u,u)=1$.
\end{enumerate}
\subsection{Double extensions of $\g_{6,1}$}\label{subsec3.3}\hfill

Let us consider $\hk=\g_{6,1}$ and choose a canonical basis $\{X_i,Z_i\}$, $1\leq i \leq 3$, of $\hk$ such that the nonzero brackets $[X_1,X_2]=Z_3$, $[X_2,X_3]=Z_1$, $[X_3,X_1]=Z_2$. If $D$ is a skew-symmetric derivation of $\hk$ then the matrix of $D$ in the given basis is:
$$D=\begin{pmatrix} A & 0 \\ B & -A^t\end{pmatrix}$$
where $A$ is a $3\times 3$ matrix with zero trace and $B$ is a skew-symmetric $3\times 3$ matrix. Eliminating inner derivations we can assume that $B=0$. 

Let $P$ 
be a $3\times 3$ matrix with determinant 1. Change the basis by setting $\begin{pmatrix} X_1' & X_2' &  X_3'\end{pmatrix}=\begin{pmatrix} X_1 & X_2 &  X_3 \end{pmatrix}P$
and $\begin{pmatrix} Z_1' & Z_2' & Z_3'\end{pmatrix}=\begin{pmatrix} Z_1 & Z_2 &  Z_3 \end{pmatrix}(P^t)^{-1}$. It is easy to check that $\{X_i',Z_i'\}$ is still a canonical basis of $\hk$ and $[X_1',X_2']=Z_3'$, $[X_2',X_3']=Z_1'$, $[X_3',X_1']=Z_2'$. In this case, the matrix of $D$ in the basis $\{X_i',Z_i'\}$:
$$D=\begin{pmatrix} P^{-1}AP & 0 \\ 0 & -P^tA^t(P^t)^{-1}\end{pmatrix}.$$
This implies the classification of double extensions of $\g_{6,1}$ following the classification of similar equivalent classes of $A$. Since $A$ has the zero trace so we have the following cases for the matrix $A$ (up to nonzero multiples):
$$\begin{pmatrix} 0 & 0 & 0 \\ 0 & 0 & 0 \\ 0 & 0 & 0 \end{pmatrix},\ \begin{pmatrix}0 & 1 & 0 \\ 0 & 0 & 0 \\ 0 & 0 & 0 \end{pmatrix},\ \begin{pmatrix} 0 & 1 & 0 \\ 0 & 0 & 1 \\ 0 & 0 & 0 \end{pmatrix},$$ 
$$\begin{pmatrix} 0 & 0 & 0 \\ 0 & 1 & 0 \\ 0 & 0 & -1 \end{pmatrix},\ \begin{pmatrix} 1 & 1 & 0 \\ 0 & 1 & 0 \\ 0 & 0 & -2 \end{pmatrix},\ \begin{pmatrix} 1 & 0 & 0 \\ 0 & \alpha & 0 \\ 0 & 0 & -1-\alpha \end{pmatrix}$$
where $\alpha\neq 0$ and $-1$. 

For two first cases, we can check that $\g$ is a 1-step double extension. For the third, $\g$ is 5-step nilpotent and $\dim([[\g,\g],[\g,\g]])=2$ so $\g$ is indecomposable and not 1-step. We denote it by $\g_{8,2}$.

\begin{enumerate}
	\item $A=\begin{pmatrix} 0 & 0 & 0 \\ 0 & 1 & 0 \\ 0 & 0 & -1 \end{pmatrix}$ then the obtained quadratic Lie algebra $\g$ is not a 1-step double extension since $[[\g,\g],[\g,\g]]$ is of 2-dimensional. We shall prove that $\g$ is indecomposable. Indeed, assume $\g$ is decomposable. Since $\Zs(\g)=\spa\{Z_1,f\}$ is totally isotropic and $\g$ is solvable then $\g$ must be decomposable by $\g=\g_4\oplusp\g_4$. Observe that the derived ideal of $\g_4\oplusp\g_4$ is isomorphic to $\hk(1)\oplus\hk(1)$ as a Lie algebra where $\hk(1)$ is the Heisenberg Lie algebra of dimension 3. This is a contradiction because $[\g,\g]=\spa\{X_2,X_3,Z_1,Z_2,Z_3,f\}$ with Lie bracket: $[X_2,Z_2]=[Z_3,X_3]=f$, $[X_2,X_3]=Z_1$.  Therefore $\g$ is indecomposable and we denote it by $\g_{8,5}$. Another proof for $\g_{8,5}$ indecomposable by applying Proposition \ref{prop2} can be found in Remark \ref{rem3.3}.
	\item $A=\begin{pmatrix} 1 & 1 & 0 \\ 0 & 1 & 0 \\ 0 & 0 & -2 \end{pmatrix}$ and $\begin{pmatrix} 1 & 0 & 0 \\ 0 & \alpha & 0 \\ 0 & 0 & -1-\alpha\end{pmatrix}$ with $\alpha\neq 0$ and $-1$. In this case, the corresponding quadratic Lie algebras have $\dim([\g,\g])=7$ and $\dim([[\g,\g],[\g,\g]])=4$ so they are indecomposable and not 1-step. We denote them by $\g_{8,3}$ and $\g_{8,4}(\alpha)$.
	\end{enumerate}
	\begin{prop}
For all $\alpha\neq 0$, the Lie algebras $\g_{8,1}(\alpha)$ and $\g_{8,2}$ are i-isomorphic.
\end{prop}
\begin{proof}
Let $Z=uZ_2+vf$ be a nonzero central element in $\g_{8,1}(\alpha)$. It
is obvious that $Z$ isotropic. One has
$Z^\bot=\spa\{X_1,T,Z_1,Z_2,Y,f,X\}$ with $X:=vX_2-ue$. Denote by $L=[X,X_1]=[vX_2-ue,X_1]=-vT-u\alpha Y$ then $[L,X_1]=-vZ_2+u\alpha^2f$ and $[L,X]=(v^2+u^2\alpha^2)Z_1$. So if we choose $u\neq 0$ and $v=\iota \alpha u$ then $Z^\bot/\FF Z$ is identified as $\g_{6,1}$. Therefore, for every $\alpha\neq 0$, the Lie algebras $\g_{8,1}(\alpha)$ and $\g_{8,2}$ are i-isomorphic.
\end{proof}
\subsection{Double extensions of $\g_{6,2}(\lambda)$}\label{subsec3.4}\hfill

Let us consider $\hk=\g_{6,2}(\lambda)$ with a basis
$\{X,X_1,X_2,Z_1,Z_2,Z\}$ such that the nonzero bilinear forms
$B(X,Z)=B(X_1,Z_1)=B(X_2,Z_2)=1$ and the nonzero Lie brackets
$[X,X_1]=X_1$, $[X,X_2]=\lambda X_2$, $[X,Z_1]=-Z_1$,
$[X,Z_2]=-\lambda Z_2$, $[X_1,Z_1]=Z$, $[X_2,Z_2]=\lambda Z$ where
$\lambda\in\FF$, $0<\left|\lambda\right|\leq 1$. By a straightforward
computation, if $D$ is a skew-symmetric derivation of $\hk$ then the
matrix of $D$ with respect to the given basis is:
$$ D=\begin{pmatrix} 0 & 0 & 0 & 0 & 0 & 0  \\ y & a & c & 0 & 0 & 0 \\ z & b & d & 0 & 0 & 0 \\ t & 0 & 0 & -a & -b & 0\\ h & 0 & 0 & -c & -d & 0 \\ 0 & -t & -h & -y & -z & 0\end{pmatrix}.$$

Now we consider two following cases: $\lambda=1$ and $\lambda\neq\pm 1$ (note that the quadratic Lie algebras corresponding to the cases of $\lambda=\pm 1$ are i-isomorphic).
\begin{enumerate}
	\item If $\lambda= 1$ then one proves that any skew-symmetric outer derivation has the matrix $D=\begin{pmatrix} 0 & 0 & 0 & 0 & 0 & 0  \\ 0 & a & c & 0 & 0 & 0 \\ 0 & b & d & 0 & 0 & 0 \\ 0 & 0 & 0 & -a & -b & 0\\ 0 & 0 & 0 & -c & -d & 0 \\ 0 & 0 & 0 & 0 & 0 & 0\end{pmatrix}$. Denote by $A=\begin{pmatrix} a & c\\ b & d \end{pmatrix}$. If $P$ is a $2\times 2$ matrix such that its determinant is 1. Set $\begin{pmatrix} X_1' & X_2' \end{pmatrix}=\begin{pmatrix} X_1 & X_2 \end{pmatrix}P$
and $\begin{pmatrix} Z_1' & Z_2' \end{pmatrix}=\begin{pmatrix} Z_1 & Z_2 \end{pmatrix}(P^t)^{-1}$. It is easy to check that $B(X_i',Z_j')=\delta_{ij}$, $[X,X_1']=X_1'$, $[X,X_2']=X_2'$, $[X,Z_1']=-Z_1'$, $[X,Z_2']=-Z_2'$ and $[X_1',Z_1']=[X_2',Z_2']=Z$. In this case, the matrix of $D$ in the basis $\{X,X_i',Z_i',Z\}$ has form:
$$D=\begin{pmatrix} 0 & 0 & 0 & 0 \\ 0& P^{-1}AP & 0 & 0 \\ 0 & 0 &  -P^tA^t(P^t)^{-1} & 0 \\ 0 & 0 & 0 & 0 \end{pmatrix}$$

Since every matrix is similar to a matrix in Jordan form so we consider following matrices $A$ (up to nonzero multiples): $\begin{pmatrix} 0 & 0  \\ 0 & 0 \end{pmatrix}$, $\begin{pmatrix} 1 & 0  \\ 0 & \beta\end{pmatrix}$, 
$\begin{pmatrix} 0 & 1  \\ 0 & 0 \end{pmatrix}$, $\begin{pmatrix} 1 & 1  \\ 0 & 1 \end{pmatrix}$ where $\beta\in\FF$. We can verify that two first cases correspond to $\g$ decomposable.

For the case of $A=\begin{pmatrix} 0 & 1  \\ 0 & 0 \end{pmatrix}$, $\g$ is exactly the Lie algebra $\g_{8,5}$ appeared in Subsection \ref{subsec3.3}. Indeed, $\g=\spa\{e,X_2,Z_1,f,Z_2,X_1\}\oplus \FF X\oplus\FF Z$ the double extension of $\g_{6,1}=\spa\{e,X_2,Z_1,f,Z_2,X_1\}$ by $$ D=\begin{pmatrix} 0 & 0 & 0 & 0 & 0 & 0  \\ 0 & 1 & 0 & 0 & 0 & 0 \\ 0 & 0 & -1 & 0 & 0 & 0 \\ 0 & 0 & 0 & 0 & 0 & 0\\ 0 & 0 & 0 & 0 & -1 & 0 \\ 0 & 0 & 0 & 0 & 0 & 1\end{pmatrix}.$$ 

For the remaining case, we will denote it by $\g_{8,6}$ and prove that it is indecomposable. Indeed, assume $\g_{8,6}$ is decomposable then it must be decomposed by $\g_4\oplusp\g_4$. By a way similar to what we have done for the Lie algebra $\g_{8,5}$ above, we can conclude that $\g_{8,6}$ indecomposable. Another proof by applying Proposition \ref{prop2} can be found in Remark \ref{rem3.3}. 
	\item $\lambda\neq\pm 1$ then it is easy to check that $b=c=0$. 
	If $a=d=0$ then $\g$ is decomposable. If $a=0$ and $d\neq 0$ (it is similarly done for $a\neq 0$ and $d=0$) then we can choose $d=1$. 
	By changing of basis $X:=X-\lambda e$, $f:=f+\lambda Z$ and keep the other we get $\g=\g_4\oplusp\g_4$ decomposable.
	
	If $a$ and $d\neq 0$, dividing by $a$ then  $ D=\begin{pmatrix} 0 & 0 & 0 & 0 & 0 & 0  \\ 0 & 1 & 0 & 0 & 0 & 0 \\ 0 & 0 & \alpha & 0 & 0 & 0 \\ 0 & 0 & 0 & -1 & 0 & 0\\ 0 & 0 & 0 & 0 & -\alpha & 0 \\ 0 & 0 & 0 & 0 & 0 & 0\end{pmatrix}$ with $\alpha \neq 0$. If $\alpha=\lambda$ then $\g$ is decomposable. If $\alpha\neq \lambda$ then by changing the basis $e:=\frac{\lambda}{\lambda-\alpha}\left(e-\frac{\alpha}{\lambda}X\right)$, $Z:=Z+\frac{\alpha}{\lambda}f$, $f:=\frac{\lambda-\alpha}{\lambda}f$ and keep the other we get $\alpha=0$ where $\g$ is decomposable.
\end{enumerate}
\begin{prop}
The Lie algebras $\g_{8,5}$ and $\g_{8,6}$ are i-isomorphic.
\end{prop}
\begin{proof}
Firstly, let us describe $\g_{8,6}=\spa\{e,X,X_2,X_1,Z_1,Z_2,Z,f\}$ by the Lie brackets: $[e,X_1]=X_1$, $[e,X_2]=X_1 + X_2$, $[e,Z_1] = -Z_1-Z_2$, $[e,Z_2] = -Z_2$, $[X,X_1] = X_1$, $[X,X_2] = X_2$, $[X,Z_1] = -Z_1$, $[X,Z_2] = -Z_2$, $[X_1,Z_1] = [X_2,Z_2] = Z+f$ and $[X_2,Z_1]=f$. Denote by $K=Z+f$ and $s=e-X$ then $K^\bot/\FF K$ is the the 6-dimensional Lie algebra with basis $\{s,X_2,X_1,Z_1,Z_2,L\}$ where $L$ is the image of $Z$ modulo $K$ and nonzero brackets $[s,X_2]=X_1$, $[s,Z_1] = -Z_2$, $[X_2,Z_1] = -L$. We recognize that is $\g_{6,1}$. Moreover, the adjoint action of $X$ on the basis $\{s,X_2,-Z_1\}$ is the matrix $\begin{pmatrix} 0 & 0 & 0 \\ 0 & 1 & 0 \\ 0 & 0 & -1 \end{pmatrix}$ so $\g_{8,6}$ is i-isomorphic to $\g_{8,5}$.
\end{proof}
\begin{rem}\label{rem3.3}\hfill

\begin{itemize}

\item The Lie algebras $\g_{8,5}$, $\g_{8,6}$ are
  double extensions of the quadratic vector space
  $\FF^4=\spa\{X_1,X_2,Z_1,Z_2\}$ by the pair $(C,D_1)$ and $(C,D_2)$ respectively where \[C=\ad (X)=\begin{pmatrix}
  1 & 0 & 0 & 0 \\ 0 & 1 & 0 & 0 \\ 0 & 0 & -1 & 0 \\ 0 & 0 & 0 &
  -1\end{pmatrix}\] is invertible and diagonalizable,
  $D_1=\begin{pmatrix} 0 & 1 & 0 & 0 \\ 0 & 0 & 0 & 0 \\ 0 & 0 & 0 & 0
  \\ 0 & 0 & -1 & 0\end{pmatrix}$ and $D_2=\begin{pmatrix} 1 & 1 & 0 & 0
  \\ 0 & 1 & 0 & 0 \\ 0 & 0 & -1 & 0 \\ 0 & 0 & -1 &
  -1\end{pmatrix}$ with respect to the basis $\{X_1,X_2,Z_1,Z_2\}$. 
	Note that $D_1=-C+D_2$ then by
  Proposition \ref{prop2}, the Lie algebras $\g_{8,5}$ and $\g_{8,6}$ is i-isomorphic. 
	\item Observe that for the Lie algebra $\g_4\oplusp\g_4$, by a suitable basis changing, we can write $\g_4\oplusp\g_4$ as a double extension of $\FF^4$ by the pair $(C,D)$ where $C=\begin{pmatrix}
  1 & 0 & 0 & 0 \\ 0 & 1 & 0 & 0 \\ 0 & 0 & -1 & 0 \\ 0 & 0 & 0 &
  -1\end{pmatrix}$ and $D=\begin{pmatrix}
  1 & 0 & 0 & 0 \\ 0 & 0 & 0 & 0 \\ 0 & 0 & -1 & 0 \\ 0 & 0 & 0 &
  0\end{pmatrix}$ with respect to a canonical basis. Both of them are diagonalizable so $\g_4\oplusp\g_4$ can not be isomorphic to $\g_{8,5}$. As a consequence, The quadratic Lie algebras $\g_{8,5}$ are indecomposable.
	\end{itemize}

\end{rem}
\subsection{Double extensions of $\g_{6,3}$}\label{subsec3.5}\hfill

Finally, let us consider $\hk=\g_{6,3}$. Choose a basis $\{X,X_1,X_2,Z_1,Z_2,Z\}$ satisfying $B(X,Z)=1$, $B(X_i,Z_j)=\delta_{ij}$, the others trivial and the Lie brackets: $[X,X_1]=X_1$, $[X,X_2]=X_1+X_2$, $[X,Z_1]=-Z_1-Z_2$, $[X,Z_2]=-Z_2$, $[X_1,Z_1]=[X_2,Z_1]=[X_2,Z_2]=Z$. By a straightforward computation, the matrix of a skew-symmetric derivation of $\hk$ with respect to the given basis is given by:
\[D=\begin{pmatrix} 0 & 0 & 0 & 0 & 0 & 0  \\ x & a & b & 0 & 0 & 0 \\ y & 0 & a & 0 & 0 & 0 \\ z & 0 & 0 & -a & 0 & 0\\ t & 0 & 0 & -b & -a & 0 \\ 0 & -z & -t & -x & -y & 0\end{pmatrix}.\]
Eliminating inner derivations, we choose $a$, $x$, $y$, $z$ and $t$ zero. 
If $b=0$ then $\g$ is decomposable. If $b\neq 0$ then it is reduced to consider
\[D=\begin{pmatrix} 0 & 0 & 0 & 0 & 0 & 0  \\ 0 & 0 & 1 & 0 & 0 & 0 \\ 0 & 0 & 0 & 0 & 0 & 0 \\ 0 & 0 & 0 & 0 & 0 & 0\\ 0 & 0 & 0 & -1 & 0 & 0 \\ 0 & 0 & 0 & 0 & 0 & 0\end{pmatrix}.\]
In this case, $\g$ is the double extension of $\hk'=\g_{6,1}=\spa\{e,X_2,Z_1,Z+f,Z_2,X_1\}$ by the derivation
\[D=\ad(X-e)=\begin{pmatrix} 0 & 0 & 0 & 0 & 0 & 0  \\ 0 & 1 & 0 & 0 & 0 & 0 \\ 0 & 0 & -1 & 0 & 0 & 0 \\ 0 & 0 & 0 & 0 & 0 & 0\\ 0 & 0 & 0 & 0 & -1 & 0 \\ 0 & 0 & 0 & 0 & 0 & 1\end{pmatrix}.\]
This is the Lie algebra $\g_{8,5}$ has just been considered in Subsection \ref{subsec3.3}.
\begin{prop}
In summary, solvable quadratic Lie algebras of dimension 8 are
separated in the following non i-isomorphic families:

\begin{enumerate}
	\item $\g_7\oplusp\FF$ with $\g_7$ a solvable quadratic Lie algebras of dimension 7.
	\item $\g_4\oplusp\g_4$.
	\item indecomposable 1-step double extensions by linear maps $A\in\ok(6)$ with totally isotropic kernel up to conjugation and up to nonzero multiples. In this case, the classification has been given in \cite{DPU12}.
	\item the double extensions $\g_{8,2}$, $\g_{8,3}$, $\g_{8,4}(\alpha)$ of $\g_{6,1}$ by $D=\begin{pmatrix} A & 0 \\ 0 & -A^t\end{pmatrix}$ where $A=\begin{pmatrix} 0 & 1 & 0 \\ 0 & 0 & 1 \\ 0 & 0 & 0 \end{pmatrix}
$, $\begin{pmatrix} 1 & 1 & 0 \\ 0 & 1 & 0 \\ 0 & 0 & -2 \end{pmatrix}$, $\begin{pmatrix} 1 & 0 & 0 \\ 0 & \alpha & 0 \\ 0 & 0 & -1-\alpha \end{pmatrix}$, $\alpha\neq 0$ and $-1$, respectively
in a canonical basis $\{X_i,Z_i\}$, $1\leq i \leq 3$, of $\g_{6,1}$ such that the nonzero brackets $[X_1,X_2]=Z_3$, $[X_2,X_3]=Z_1$, $[X_3,X_1]=Z_2$. 
	\item the double extension $\g_{8,5}$ of $\g_{6,2} (1)$ by  $$D=\begin{pmatrix} 0 & 0 & 0 & 0 \\ 0& A & 0 & 0 \\ 0 & 0 &  -A^t & 0\\ 0 & 0 & 0 & 0\end{pmatrix}$$
where $A= \begin{pmatrix} 0 & 1  \\ 0 & 0 \end{pmatrix}$
 respectively in a basis $\{X,X_i,Z_i,Z\}$, $1\leq i \leq 2$, 
of $\g_{6,2} (1)$ such that the non trivial bilinear forms $B(X,Z)=1$, $B(X_i,Z_j)=\delta_{ij}$, 
the nonzero Lie brackets $[X,X_1]=X_1$, $[X,X_2]=X_2$, $[X,Z_1]=-Z_1$, $[X,Z_2]=-Z_2$ and $[X_1,Z_1]=[X_2,Z_2]=Z$.
		
\end{enumerate}
\end{prop}

\begin{rem}\label{rem3.4}\hfill

\begin{itemize}

\item The Lie algebras $\g_{8,2}$ is nilpotent
  while $\g_{8,3}$, $\g_{8,4}(\alpha)$ and $\g_{8,5}$ are solvable nonnilpotent.

\item The Lie algebras $\g_{8,3}$ and $\g_{8,4}(\alpha)$ are double
  extensions of $\g_{6,1}$ by invertible skew-symmetric derivations
  and the derivation corresponding to $\g_{8,4}(\alpha)$ is
  diagonalizable on the center of $\g_{6,1}$. Hence, applying
  Corollary \ref{cor1}, $\g_{8,3}$ and $\g_{8,4}(\alpha)$ are not
  i-isomorphic. Two Lie algebras $\g_{8,4}(\alpha)$ and
  $\g_{8,4}(\alpha')$ are i-isomorphic if and only if \linebreak
  $(1,\alpha,$$ -1-\alpha) =(1,\alpha',-1-\alpha')$ up to nonzero
  multiples and up to coordinates permutation.

	\end{itemize}

\end{rem}

\begin{rem}\hfill
\begin{enumerate}
	\item In the classification of indecomposable solvable quadratic Lie algebras of dimension 8, we can begin with a reduced solvable quadratic Lie algebra
$\g$ of dimension 8
. Since
$[\g,\g]/\Zs(\g)$ is nilpotent then we have three situations: $[\g,\g]/\Zs(\g)$ is Abelian, isomorphic to $\g_5\oplusp\FF$ or isomorphic to $\g_{6,1}$. For the case $[\g,\g]/\Zs(\g)$ Abelian, the reader can find a classification in \cite{KO04} and \cite{KO06} by cohomology. In the
case $[\g,\g]/\Zs(\g)$ isomorphic to $\g_5\oplusp\FF$, it is obvious that
$\dim(\Zs(\g))=1$, $\g$ is a double extension of $\g_5\oplusp\FF$ by a skew-symmetric derivation which is invertible on the center of $\g_5\oplusp\FF$. Following the calculation of skew-symmetric derivations in Subsection \ref{subsec3.2}, this can not happen. For the last, $\dim(\Zs(\g))=1$, $\g$ is a double extension of $\g_{6,1}$ by a skew-symmetric derivation which is invertible on the center of $\g_{6,1}$. It has been indicated in Subsection \ref{subsec3.3}.
	\item Combining with a classification result in \cite{BE14}, we obtain a classification of not necessarily indecomposable, non solvable, quadratic Lie algebras of dimensions $\leq 11$
\end{enumerate}
\end{rem}
\bibliographystyle{amsxport}

\end{document}